\documentclass[12pt]{amsart}
\usepackage[margin=3cm, marginpar=3cm]{geometry}

\usepackage{amssymb}
\usepackage{braket}
\usepackage{mathrsfs}
\usepackage{ifthen}
\usepackage{graphicx}

\usepackage{comment} 
\usepackage{mleftright}
\usepackage[pagebackref,hypertexnames=false]{hyperref} 
\usepackage{cleveref}
 \usepackage[all]{xy}
 \usepackage{amscd}
 \usepackage{color}
 \usepackage{enumitem}
\newlist{steps}{enumerate}{1}
\setlist[steps, 1]{label = Step \arabic*:}
\usepackage[fontsize=10]{scrextend}
\usepackage[pagebackref,hypertexnames=false]{hyperref} 
\usepackage[alphabetic,backrefs,msc-links]{amsrefs}
\usepackage{adjustbox}
\usepackage{subcaption}

\makeatletter
\DeclareRobustCommand\widecheck[1]{{\mathpalette\@widecheck{#1}}}
\def\@widecheck#1#2{%
   \setbox\z@\hbox{\m@th$#1#2$}%
   \setbox\tw@\hbox{\m@th$#1%
      \widehat{%
         \vrule\@width\z@\@height\ht\z@
         \vrule\@height\z@\@width\wd\z@}$}%
   \dp\tw@-\ht\z@
   \@tempdima\ht\z@ \advance\@tempdima2\ht\tw@ \divide\@tempdima\thr@@
   \setbox\tw@\hbox{%
      \raise\@tempdima\hbox{\scalebox{1}[-1]{\lower\@tempdima\box\tw@}}}%
   {\ooalign{\box\tw@ \cr \box\z@}}}
\makeatother

\theoremstyle{plain}
\newtheorem{thm}{Theorem}[section]
\crefname{thm}{Theorem}{Theorems}
\Crefname{thm}{Theorem}{Theorems}
\newtheorem{prop}[thm]{Proposition}
\crefname{prop}{Proposition}{Propositions}
\Crefname{prop}{Proposition}{Propositions}
\newtheorem{lem}[thm]{Lemma}
\crefname{lem}{Lemma}{Lemmas}
\Crefname{lem}{Lemma}{Lemmas}
\newtheorem{cor}[thm]{Corollary}
\crefname{cor}{Corollary}{Corollaries}
\Crefname{cor}{Corollary}{Corollaries}

\crefname{claim}{Claim}{Claims}
\Crefname{claim}{Claim}{Claims}

\crefname{property}{Property}{Properties}
\Crefname{property}{Property}{Properties}

\crefname{problem}{Problem}{Problems}
\Crefname{problem}{Problem}{Problems}
\newtheorem{ques}[thm]{Question}
\crefname{ques}{Question}{Questions}
\Crefname{ques}{Question}{Questions}

\theoremstyle{definition}
\newtheorem{defn}[thm]{Definition}
\crefname{defn}{Definition}{Definitions}
\Crefname{defn}{Definition}{Definitions}

\crefname{notation}{Notation}{Notations}
\Crefname{notation}{Notation}{Notations}

\crefname{convention}{Convention}{Conventions}
\Crefname{convention}{Convention}{Conventions}

\crefname{cond}{Condition}{Conditions}
\Crefname{cond}{Condition}{Conditions}

\crefname{assum}{Assumption}{Assumptions}
\Crefname{assum}{Assumption}{Assumptions}

\crefname{conj}{Conjecture}{Conjectures}
\Crefname{conj}{Conjecture}{Conjectures}
\crefname{ques}{Question}{Questions}
\Crefname{ques}{Question}{Questions}

\theoremstyle{remark}
\newtheorem{rem}[thm]{Remark}
\crefname{rem}{Remark}{Remarks}
\Crefname{rem}{Remark}{Remarks}

\crefname{ex}{Example}{Examples}
\Crefname{ex}{Example}{Examples}

\crefname{section}{Section}{Sections}
\Crefname{section}{Section}{Sections}
\crefname{subsection}{Subsection}{Subsections}
\Crefname{subsection}{Subsection}{Subsections}
\crefname{figure}{Figure}{Figures}
\Crefname{figure}{Figure}{Figures}

\newtheorem*{acknowledgement}{Acknowledgement}

\newcommand{\Z}{\mathbb{Z}}

\newcommand{\Q}{\mathbb{Q}}

\newcommand{\C}{\mathbb{C}}

\newcommand{\ctext}[1]{\raise0.2ex\hbox{\textcircled{\scriptsize{#1}}}}

\def\det{\operatorname{det}}

\def\dim{\operatorname{dim}}

\newcommand{\mbar}[1]{{\ooalign{\hfil#1\hfil\crcr\raise.167ex\hbox{--}}}}

\def\wt{\widetilde}

\title{On the slice-torus invariant $q_M$ from $\Z_2$-equivariant Seiberg--Witten Floer cohomology}

\author{Nobuo Iida}
\address{Tokyo Institute of Technology, Ookayama, Meguro-ku, Tokyo}
\email{iida.n.ad@m.titech.ac.jp}

\author{Taketo Sano}
\address{RIKEN iTHEMS, Wako, Saitama 351-0198, Japan}
\email{taketo.sano@riken.jp}

\author{Kouki Sato}
\address{Meijo University, Tempaku, Nagoya 468-8502, Japan}
\email{satokou@meijo-u.ac.jp}

\author{Masaki Taniguchi} 
\address{Department of Mathematics, Graduate School of Science, Kyoto University, Kitashirakawa Oiwake-cho, Sakyo-ku, Kyoto 606-8502, Japan}
\email{taniguchi.masaki.7m@kyoto-u.ac.jp}

\begin{document}

\begin{abstract}
We show that Iida--Taniguchi's $\Z$-valued slice-torus invariant \( q_M \) cannot be realized as a linear combination of Rasmussen's \( s \)-invariant,  Ozsv\'ath--Szab\'o's \( \tau \)-invariant, all of the $\mathfrak{sl}_N$-concordance invariants ($N \geq 2$), Baldwin--Sivek's instanton $\tau$-invariant,  Daemi--Imori--Sato--Scaduto--Taniguchi's instanton $\tilde{s}$-invariant and Sano--Sato's Rasmussen type invariants $\tilde{ss}_c$. 
\end{abstract}

\maketitle

\section{Introduction}
Link homology theories have become central to modern knot theory, particularly from a four-dimensional perspective. These include prominent examples such as Khovanov homology, Heegaard Floer homology, monopole Floer homology, and instanton Floer homology. Various link homology theories yield numerous concordance invariants, among which one of the simplest and most extensively studied classes is the family of \textit{slice-torus invariants}. A slice-torus invariant (\cite{Li04, Le14}) is a real-valued function \( f \) defined on the smooth knot concordance group that satisfies the following properties: for knots $K,K'$ in $S^3$,
\begin{itemize}
    \item[(i)] \( |f(K)| \leq g_4(K) \),
    \item[(ii)] \( f(K \# K') = f(K) + f(K') \), and
    \item[(iii)] \( f(T(p,q)) = \frac{1}{2} (p-1)(q-1) \),
\end{itemize}
where \( g_4(K) \) denotes the smooth slice genus and \( T(p,q) \) is the \((p,q)\)-positive torus knot for coprime integers \( p,q \geq 2 \).

Examples of slice-torus invariants arise from several theories, including Heegaard Floer theory \cite{OS03}, Khovanov homology theory \cite{Ra10, Lobb09, Wu09,  LS14, SS22, Lo12, Le14, LL:2016}, instanton Floer theory \cite{GLW19, BS21, DISST22}\footnote{ Ghosh--Li--Wong \cite{GLW19} proved Baldwin--Sivek's instanton tau invariant $\tau^\#$ coicides with the concordance invariant $\tau_I \in \Z$ which comes from the Alexander decomposition of a variant of sutured instanton homology $\underline{KHI}^-$. 
This ensures $\tau^\#$ is actually integer-valued.   }, and Seiberg--Witten Floer theory \cite{IT24}. Once a slice-torus invariant is obtained, it immediately provides a solution to the Milnor conjecture on the slice genus of torus knots and reproves the existence of exotic \(\mathbb{R}^4\) by showing the existence of knots that are topologically slice but not smoothly slice. Moreover, for a large class of knots, including quasipositive and alternating knots, the values of all slice-torus invariants coincide. This observation underlies the conjecture that Ozsv\'ath--Szab\'o's \(\tau\)-invariant and Rasmussen's \(s\)-invariant are equal.

Notably, in \cite{HO08}, Hedden and Ording demonstrated that Rasmussen's \(s\)-invariant and Ozsv\'ath--Szab\'o's \(\tau\)-invariant are not identical. Similarly, Lewark \cite{Le14} proved the linear independence of \(\tau\), \(s\), and Rasmussen-type invariants \(s_N\) derived from $\mathfrak{sl}_N$-Khovanov--Rozansky homology theory. See also \cite{MPP07, LS14, LC24, Sc23} for linear independence of Rasmussen invariants with different coefficients. 
There has also been significant progress in studying the general behavior of slice-torus invariants \cite{Li04, Le14, FLL22, FLL24}.

In \cite{IT24}, the first and fourth authors introduced a \(\mathbb{Z}\)-valued slice-torus invariant \( q_M(K) \) arising from \(\mathbb{Z}_2\)-equivariant Seiberg--Witten theory applied to double branched covering spaces of knots, conjecturally equal to the Heegaard Floer $q_\tau$-invariant introduced by Hendricks--Lipshitz--Sarkar \cite{HLS16} with a signature correction term. 
A natural question is whether \( q_M \) coincides with other known slice-torus invariants. 
In this paper, we address this question by proving the following:

\begin{thm}\label{main}
Let us denote by $s$, $\tau$, $s_{\partial \omega, \alpha}$, $\tau^\#$, $\tilde{s}$, and $\tilde{ss}_c$ the Rasmussen invariant \cite{Ra10}, the Ozsv\'ath--Szab\'o $\tau$-invariant \cite{OS03}, the $\mathfrak{sl}_N$-concordance invariant ($N \geq 2$) \cite{Lobb09, Wu09, LL:2016} with any separable potential $\partial \omega$ equipped with any root $\alpha$, the instanton \( \tau \)-invariant  \cite{BS21}, the instanton $\tilde{s}$-invariant  \cite{DISST22}, and Sano--Sato's Rasmussen type invariant \cite{SS22} for any PID $R$ with a prime element $c$ respectively. 
Then we have
\[
q_M(9_{42})=-1 \text{ while } s(9_{42}) = \tau (9_{42}) =s_{\partial \omega, \alpha}( 9_{42})= \tau^\# (9_{42}) = \tilde{s} (9_{42}) =\tilde{ss}_c(9_{42}) =  0,
\]
where the convention of $9_{42}$ follows the knotinfo \cite{knotinfo}. 
In particular, the invariant $q_M$ cannot be realized as a linear combination of $s$, $\tau$, $s_{\partial \omega, \alpha}$, $\tau^\#$, $\tilde{s}$, and $\tilde{ss}_c$. 

\end{thm}

\begin{rem}
Note that Baraglia showed that the concordance invariant \(\theta \in \Z_{\geq 0} \) from \(S^1 \times \mathbb{Z}_2\)-equivariant Seiberg--Witten theory \cite{Ba22, BH} satisfies
\[
\theta(9_{42}) = 0 \quad \text{and} \quad \theta(-9_{42}) = 1.
\]
In general, for any knot \( K \subset S^3 \),
\[
q_M(K) \leq \theta(K) \leq g_4(K)
\]
holds. 

The first inequality was proved by the first and fourth authors in \cite[Theorem 1.11]{IT24}, while the second inequality was proved by Baraglia in \cite[Theorem 1.4]{Ba22}. It is known that \(g_4(9_{42}) = 1\), so in summary, \(q_M\) and \(\theta\) provide the optimal 4-ball genus bound for \(9_{42}\), whereas the others in \cref{main} do not.

\end{rem}

As an immediate corollary, we have: 
\begin{cor}\label{aa} Let $n$ be a non-zero integer. 
  The $n$-fold connected sum $\#_n 9_{42}$ is not a squeezed knot. 
\end{cor}
\begin{proof} This follows from the fact that every slice torus invariant takes the same value \cite{FLL22} for squeezed knots.
\end{proof}
\begin{rem}
Note that a refinement of the Rasmussen invariant $s_+^{\operatorname{Sq}_2}$ introduced by Lipshitz--Sarkar \cite{LS14}, which uses Khovanov homotopy type with Steenrod operator, satisfies $s_+^{\operatorname{Sq}_2}(9_{42}) /2=1$. This fact has been used to prove $9_{42} $ is not squeezed. \cref{aa} might have alternative proof by showing $s_+^{\operatorname{Sq}_2}(\#_n 9_{42}) \neq 0$. 
\end{rem}

{\bf Structure of the paper:}
In \cref{proof of main}, we give a proof of \cref{main}. In \cref{qmtehta}, we briefly discuss about the backgrounds of the invariants $q_M$ and $\theta$ and summarize basic properties of them. Also, we shall discuss the values of these invariants for prime knots up to $10$ crossings. We put the tables of these values in \cref{Table}. 

\begin{acknowledgement}
We would like to thank Joshua Greene for answering to our question regarding \cite{Gr13}.  The first author acknowledges support from JSPS KAKENHI Grant Number 22J00407. The second author acknowledges support from JSPS KAKENHI Grant Numbers 23K12982, RIKEN iTHEMS Program and academist crowdfunding. 
The fourth author acknowledges partial support from JSPS KAKENHI Grant Number 22K13921.
\end{acknowledgement}

\section{Proof of \cref{main}}\label{proof of main}

\begin{figure}[t]
\begin{center}
\includegraphics[scale=0.7]{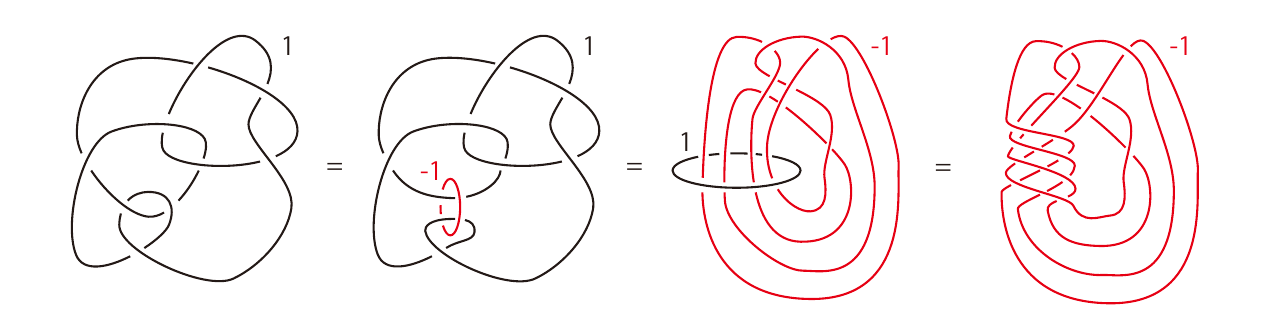}
\caption{Kirby moves which show $S^3_1(9_{42}) \cong S^3_{-1}(K)$ for some knot $K$ in $S^3$}\label{negative_definite}
\vspace{1em}
\includegraphics[scale=0.7]{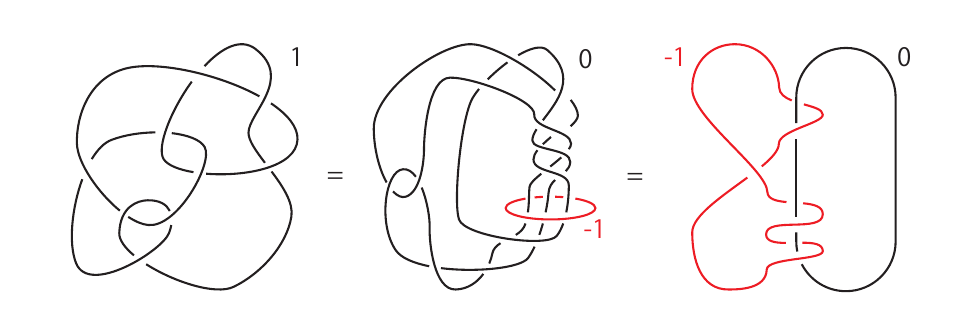}
\caption{Kirby moves which show that $S^3_1(-9_{42})$ bounds a smooth compact contractible 4-manifold}\label{contractible}
\end{center}
\end{figure}

We first prove 
that the invariants $\tau$, $\tau^{\sharp}$ and $\tilde{s}$ vanish
for $9_{42}$: 
\begin{lem}\label{vanishing}
We have 
\[
\tau (9_{42}) = \tau^\# (9_{42}) = \tilde{s} (9_{42}) = 0. 
\]
\end{lem}
\begin{proof}[Proof of \cref{vanishing}] 
 We shall use the following statements: 
\begin{itemize}
\item For any knot $K$, if $\tau  (K) >0$, then we have $d(S^3_{1}(K)) <0$.
    \item For any knot $K$, if $\tilde{s} (K) >0$, then we have $h(S^3_{1}(K)) <0$.
    \item For any knot $K$, if $\tau^\# (K) >0$, then we have $h(S^3_{1}(K)) <0$.
\end{itemize}
Here $d$ denotes the Heegaard Floer correction term and $h$ denotes the instanton Fr\o yshov invariant of oriented homology $3$-spheres introduced in \cite{Fr02} and the convention follows $h(\Sigma(2,3,5))=1$. 

The first claim follows from \cite{HW16}, see also \cite[Section 2.2]{Sa18}. 
The second fact is proven in \cite[Theorem 1.5]{DISST22} and the third fact is proven in \cite[Proposition 9.2]{BS22II}. Therefore, to show $\tau (9_{42}) = \tau^\# (9_{42}) = \tilde{s} (9_{42}) = 0$, it is sufficient to see 
\[
h(S^3_1(9_{42})) = h(S^3_1(-9_{42}))= d(S^3_1(9_{42})) = d(S^3_1(-9_{42})) = 0. 
\]
The Kirby moves described in \Cref{negative_definite} shows that 
\[
S^3_1(9_{42}) \cong S^3_{-1}(K)
\]
for some knot $K$ in $S^3$. So, $S^3_1(9_{42})$ bounds both of positive definite and negative definite 4-manifolds, hence we have $h(S^3_1(9_{42})) = d(S^3_1(9_{42})) =0$. The Kirby moves described in \Cref{contractible} shows that $S^3_1(-9_{42})$ bounds a smooth compact contractible $4$-manifold. So, we see the latter statements. 
This completes the proof. 
\end{proof}

Next, we determine Sano--Sato's concordance invariants for $9_{42}$. The link invariant $\widetilde{ss}_h$ of \cite{SS22} is defined for each non-zero non-invertible element $h$ in an integral domain $R$, and is slice-torus when $R$ is a PID and $h$ is prime. \cite[Theorem 2]{SS21} states that for the special case $(R, h) = (\mathbb{F}[H], H)$ where $\mathbb{F}$ is any field, the invariant coincides with the Rasmussen invariant $s^\mathbb{F}$ over $\mathbb{F}$. 

\begin{table}[t]
\centering
\begin{adjustbox}{max width=\textwidth}
\begin{tabular}{r|llllllllll}
$6$ & $.$ & $.$ & $.$ & $.$ & $.$ & $.$ & $\mathbb{Z}[H]$ \\
$4$ & $.$ & $.$ & $.$ & $.$ & $.$ & $\mathbb{Z}[H]$ & $.$ \\
$2$ & $.$ & $.$ & $.$ & $.$ & $\mathbb{Z}[H]$ & $.$ & $.$ \\
$0$ & $.$ & $.$ & $.$ & $\mathbb{Z}[H]^{2}$ & $\mathbb{Z}[H]$ & $.$ & $.$ \\
$-2$ & $.$ & $.$ & $\mathbb{Z}[H]$ & $.$ & $.$ & $.$ & $.$ \\
$-4$ & $.$ & $\mathbb{Z}[H]$ & $.$ & $.$ & $.$ & $.$ & $.$ \\
$-6$ & $\mathbb{Z}[H]$ & $.$ & $.$ & $.$ & $.$ & $.$ & $.$ \\
\hline
$ $ & $-4$ & $-3$ & $-2$ & $-1$ & $0$ & $1$ & $2$ \\
\end{tabular}
\hspace{2em}
\begin{tabular}{r|llllllllll}
$6$ & $.$ & $.$ & $.$ & $.$ & $.$ & $.$ & $\mathbb{Z}[H]/(H)$ \\
$4$ & $.$ & $.$ & $.$ & $.$ & $.$ & $.$ & $.$ \\
$2$ & $.$ & $.$ & $.$ & $.$ & $\mathbb{Z}[H]/(H)$ & $.$ & $.$ \\
$0$ & $.$ & $.$ & $.$ & $\mathbb{Z}[H]/(H)$ & $\mathbb{Z}[H]$ & $.$ & $.$ \\
$-2$ & $.$ & $.$ & $.$ & $.$ & $.$ & $.$ & $.$ \\
$-4$ & $.$ & $\mathbb{Z}[H]/(H)$ & $.$ & $.$ & $.$ & $.$ & $.$ \\
$-6$ & $.$ & $.$ & $.$ & $.$ & $.$ & $.$ & $.$ \\
\hline
$ $ & $-4$ & $-3$ & $-2$ & $-1$ & $0$ & $1$ & $2$ \\
\end{tabular}
\end{adjustbox}
\vspace{1em}
\caption{$\widetilde{\mathit{CBN}}(9_{42}; \Z)$ and $\widetilde{\mathit{BN}}(9_{42}; \Z).$}
\label{table:BN9_42}
\end{table}

\begin{lem}\label{ss}
    $\widetilde{ss}_h(9_{42}; R) = 0$ for any PID $R$ and prime $h \in R$.
\end{lem}

\begin{proof}
    From \cite[Lemma 4.37]{SS22}, we always have 
    \[
        \widetilde{ss}_H \leq \widetilde{ss}_h
    \]
    where the left-hand side is the (non slice-torus) invariant corresponding to the pair $(R, h) = (\Z[H], H)$. Thus it suffices to prove that 
    \[
        \widetilde{ss}_H(9_{42}) = \widetilde{ss}_H(-9_{42}) = 0.
    \]
    
    Using the program \verb|yui| \cite{YUI} developed in \cite{SS22}, the reduced Bar-Natan complex $\widetilde{\mathit{CBN}}(9_{42}; \Z)$ of $9_{42}$ over $\Z$ (in its simplified form, together with the differential matrices) can be computed by the command
    \begin{verbatim}$ ykh ckh 9_42 -t Z -c H -r -d\end{verbatim}
    whose result is displayed in the left side of \Cref{table:BN9_42}. With the differential matrices, its homology $\widetilde{\mathit{BN}}(9_{42}; \Z)$ can be easily computed as in the right side of \Cref{table:BN9_42}. Note the single free summand $\Z[H]$ in bigrading $(0, 0)$. This shows that $\widetilde{ss}_H(9_{42}) = 0$. Similarly, we can show that $\widetilde{ss}_H(-9_{42}) = 0$.
\end{proof}

Next, we determine the $\mathfrak{sl}_N$-concordance invariants for $9_{42}$. Recall from \cite{LL:2016} that a family of slice-torus invariants $s_{\partial w, \alpha}$ is given for each choice of a \textit{separable potential}, i.e.\ a degree $N$ monic polynomial $\partial w \in \C[x]$ having $N$ distinct roots in $\C$, together with a choice of a root $\alpha$ of $\partial w$. For any knot $K$, there is a spectral sequence starting from the (reduced) $\mathfrak{sl}_N$ Khovanov--Rozansky homology $H_{\mathfrak{sl}_N}$ and converging to the perturbed homology $H_{\partial w, \alpha}$ of dimension $1$. The invariant is defined as 
\[
    s_{\partial w, \alpha} = \frac{j}{2(N - 1)}
\]
where $j$ is the $\mathfrak{sl}_N$-quantum grading of the surviving generator of the $E_\infty$ term. In particular, the special case $(\partial w, \alpha) = (x^N - 1, 1)$ gives the invariant $s_N$ introduced in \cite{Lobb09, Wu09}. 

\begin{table}[t]
\centering
\begin{subtable}{.4\textwidth}
    \centering
    \begin{tabular}{r|lllll}
     $2$ & $.$ & $1$ & $.$ & $1$ & $.$ \\ 
     $0$ & $1$ & $.$ & $2$ & $.$ & $1$ \\ 
     $-2$ & $.$ & $1$ & $.$ & $1$ & $.$ \\ 
    \hline
     $a \backslash q$ & $-4$ & $-2$ & $0$ & $2$ & $4$ 
    \end{tabular}
    \caption*{$\Delta = -2$}
\end{subtable}
\hspace{2em}
\begin{subtable}{.4\textwidth}
    \centering
    \begin{tabular}{r|lllll}
     $2$ & $.$ & $.$ & $.$ & $.$ & $.$ \\ 
     $0$ & $.$ & $.$ & $1$ & $.$ & $.$ \\ 
     $-2$ & $.$ & $.$ & $.$ & $.$ & $.$ \\ 
    \hline
     $a \backslash q$ & $-4$ & $-2$ & $0$ & $2$ & $4$ 
    \end{tabular}
    \caption*{$\Delta = 0$}
\end{subtable}
\caption{$\mathcal{H}(9_{42})$}
\label{table:HOMFLY9_42}
\end{table}

\begin{lem}\label{sln}
    $s_{\partial w, \alpha}(9_{42}) = 0$ for $N \geq 2$ and any choice of $(\partial w, \alpha)$.
\end{lem}

\begin{proof}
    When $N = 2$, it is known that $s_{\partial w, \alpha} = s^\Q$ and the statement is proved above. The case for $N \geq 3$ follows from \cite[Theorem 3.14]{Chandler-Gorsky:2024} together with the computational result of the reduced HOMFLY-PT homology of $9_{42}$. First, from \cite[Proposition 3.3]{LL:2016}, $s_{\partial w, \alpha}$ is invariant under any translation of $x$, so we may assume that $\alpha = 0$. Writing $\partial w$ as 
    \[
        \partial w = x^N + a_N x^{N - 1} + \cdots + a_2 x
    \]
    with $a_2 \neq 0$, \cite[Theorem 3.14]{Chandler-Gorsky:2024} states that the first differential of the spectral sequence $H_{\mathfrak{sl}_N} \Rightarrow H_{\partial w, \alpha}$ is given by 
    \[
        d_{\partial w} = a_N d_{N - 1} + \cdots + a_2 d_1
    \]
    where each $d_k$ is the first differential of the spectral sequence given by Rasmussen in \cite{Rasmussen:2015}, starting from the (reduced) HOMFLY-PT homology $\mathcal{H}$ in the $E_1$ page and converging to the $\mathfrak{sl}_k$ homology $H_{\mathfrak{sl}_k}$. We claim that for $K = 9_{42}$, we have $\mathcal{H} \cong H_{\mathfrak{sl}_N}$ and $d_{N - 1} = \cdots = d_2 = 0$. 

    Using the program \verb|yui-kr| \cite{YUIKR} developed in \cite{Nakagane-Sano:2024}, the reduced HOMFLY-PT homology of $9_{42}$ can be computed by the command
    \begin{verbatim}$ ykr 9_42 -f delta\end{verbatim}
    whose result is given in \Cref{table:HOMFLY9_42}. Here, the triply graded homology group $\mathcal{H}(9_{42})$ is sliced by the $\Delta$-grading introduced in \cite{Chandler-Gorsky:2024}, and the structure (i.e.\ $\Q$-dimension) of each $\Delta$-slice is displayed with respect to the $(q, a)$-bigrading. Since $\mathcal{H}(9_{42})$ has $\Delta$-thickness $2$, it follows that the spectral sequence  $\mathcal{H}(9_{42}) \Rightarrow H_{\mathfrak{sl}_N}(9_{42})$ has trivial differentials and hence $\mathcal{H}(9_{42}) \cong H_{\mathfrak{sl}_N}(9_{42})$ (see \cite[Corollary 2.14]{Chandler-Gorsky:2024}). Moreover, for each $2 \leq k \leq N - 1$, the first differential $d_k$ changes the $\Delta$-grading by $2k - 2$ and the $(q, a)$ bigrading by $(2k, -2)$, so we see that $d_k = 0$. 

    Next, let us consider the spectral sequence $\mathcal{H} \Rightarrow H_{\mathfrak{sl}_1}$ for $k = 1$. The $i$-th differential $d^{(i)}_1$ changes the $\Delta$-grading by $2 - 2i$ and the $(q, a)$ bigrading by $(2i, -2i)$. Again from grading reasons, we have $d^{(i)}_1 = 0$ when $i > 1$, and also $d_1 = 0$ on the single $\Q$-summand in $\Delta$-grading $0$. This $\Q$-summand is necessarily the surviving one, and other summands must cancel out by the first differential $d_1$. 

    We conclude that the spectral sequence $\mathcal{H}(9_{42}) \cong H_{\mathfrak{sl}_N}(9_{42}) \Rightarrow H_{\partial w, \alpha}(9_{42}) \cong \Q$ with first differential $d_{\partial w} = a_2 d_1$ collapses after the first page and the quantum grading of the surviving generator is $0$.
\end{proof}

Now, we only need to see $q_M(9_{42})=-1$ to prove \cref{main} from \cref{vanishing}, \cref{ss} and \cref{sln}.

{\em Proof of \cref{main}.}
This is an immediate consequence of the fact that the $\Z_2$-branched cover  $\Sigma_2(9_{42})$ is an L-space, which is observed by Greene in \cite{Gr13}.
(However, notice that $9_{42}$ is not quasi-alternating.)
In \cite[Theorem 1.10]{IT24}, it is proved that 
\[
q_M(K)=-\frac{\sigma(K)}{2}
\]
when $\Sigma_2(K)$ is an L-space.
Since $\sigma(9_{42})=2$ we obtain the desired result.
Here our convention of the knot signature $\sigma$ follows that of \cite{IT24}, i.e. $\sigma(T(2,3)) = -2$.
\qed

In our argument above, we used the fact that  $\Sigma_2(9_{42})$ is an L-space over $\mathbb{F}_2$.
Since details of the proof are not given in \cite{Gr13}, we give a brief computation to prove it. 
\\
\begin{prop}\label{oo}(Greene \cite{Gr13}).
The 3-manifold $\Sigma_2(9_{42})$ is an L-space over the coefficient $\mathbb{F}_2$.
\end{prop}

\begin{figure}[htbp]
\begin{center}
\includegraphics[width=50mm]{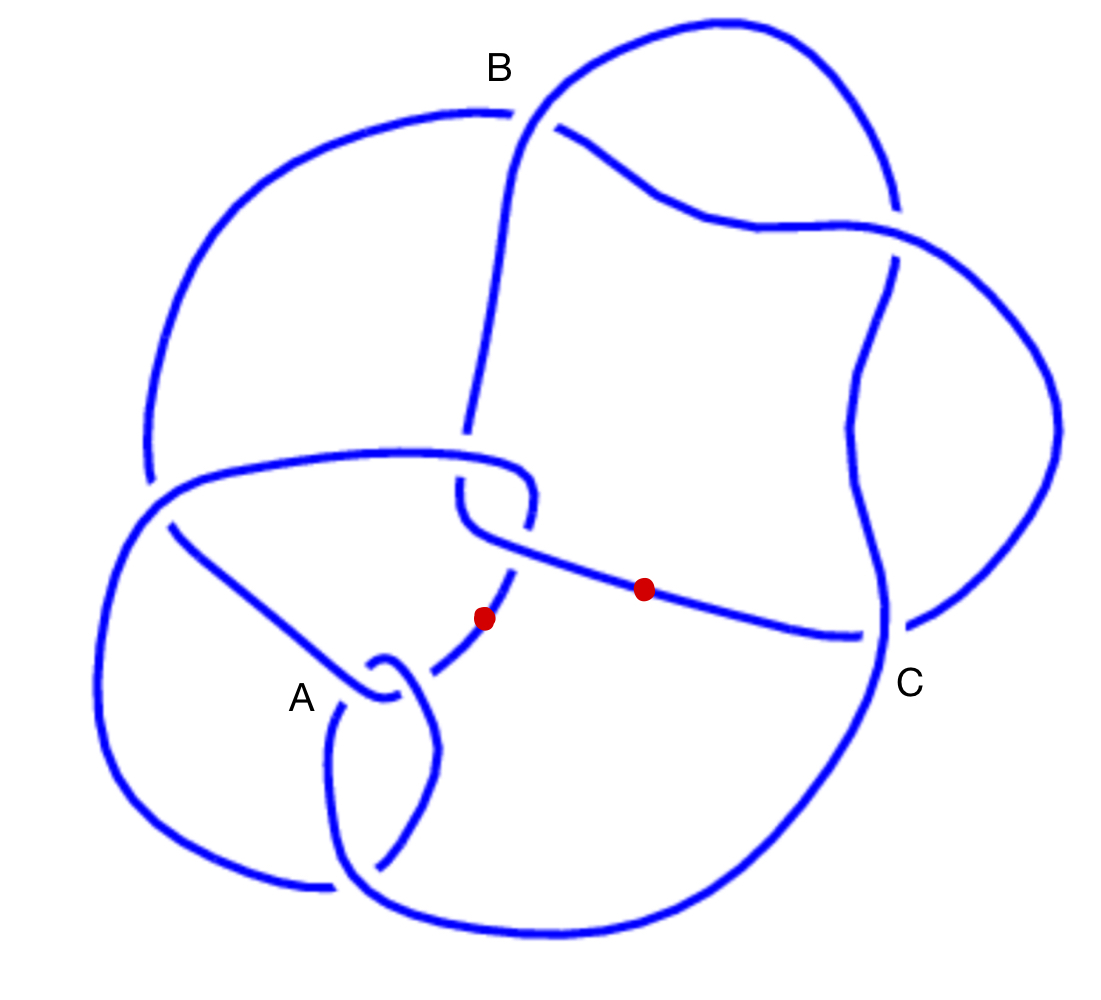}
\caption{A diagram of $9_{42}$ from the KnotInfo \cite{knotinfo}. 
We added names $A, B, C$ for three crossings and markings.}\label{942}
\end{center}
\end{figure}
\begin{proof}[Proof of \cref{oo}]
KnotFolio \cite{knotfolio} is helpful for the authors to check the following argument.
As the two smoothings at the crossing $A$ in the figure \ref{942}, we obtain the knot $8_{19}$ and the link $7n1$. Note that the determinants of $8_{19}$ and $7n1$ are $3$ and $4$ respectively. 
The not $8_{19}$ is nothing but the torus knot $T_{3, 4}$, and since its branched double cover $\Sigma_2(T_{3, 4})=\Sigma(2, 3, 4)$ has a metric with positive scalar curvature (See \cite{milnor19753} for example), it is an L-space.
We can check that $7n1$ is two-fold quasi-alternating (TQA)\cite{SS18} \cite[Section 3]{IsTu24} by smoothing at $B$ and $C$, with the marking described by dots in the diagram.
More precisely we obtain the following:
\[
7n1\xrightarrow{\text{smoothing at }B} \dot{U}_2 \, \text{ and } \, L6n1
\]
\[
L6n1\xrightarrow{\text{smoothing at }C}\dot{U}_2 \,\text{ and } \, L4a1
\]
where $\dot{U}_2$ is the two-component unlink with a dot on each component.
Since $L4a1$ is a non-split alternating link, which is TQA, 
$L6n1$ and thus $7n1$ are also TQA.
Thus the branched double covering  $\Sigma_2(7n1)$ is also an $L$-space over $\mathbb{F}_2$ coefficient by \cite[Corollary 1]{SS18}.
Now by applying 
the long exact sequence for the branched double coverings \cite{OS05} for the smoothing
\[
9_{42}\xrightarrow{\text{smoothing at }A}8_{19}\,\text{ and } \,7n1,
\]
we obtain 
\[
\dim_{\mathbb{F}_2}\widehat{HF}(\Sigma_2(9_{42}))\leq \dim_{\mathbb{F}_2}\widehat{HF}(\Sigma_2(8_{19}))+\dim_{\mathbb{F}_2}\widehat{HF}(\Sigma_2(7n1))=3+4=7=|H_1(\Sigma_2(9_{42}))|.
\]
On the other hand, for any closed oriented 3-manifold $Y$, 
\[
\dim_{\mathbb{F}_2}\widehat{HF}(Y)\geq \chi(\widehat{HF}(Y))=|H_1(Y; \Z)|
\]
holds and this implies
\[
\dim_{\mathbb{F}_2}\widehat{HF}(\Sigma_2(9_{42}))\geq |H_1(\Sigma_2(9_{42}))|.
\]
Thus, $\dim_{\mathbb{F}_2}\widehat{HF}(\Sigma_2(9_{42}))=|H_1(\Sigma_2(9_{42}))|$ and therefore
$\Sigma_2(9_{42})$ is an $L$-space over $\mathbb{F}_2$  as well.

\end{proof}

\begin{rem}
    As we discussed in the proof of \cref{vanishing}, the slice torus invariants $\tau$, $\tau^\#$ and $\tilde{s}$ are related to Heegaard Floer correction term $d$ and instanton Fr\o yshov invariant $h$ of surgeries. This is one reason that the equalities 
    \[
    \tau =\tilde{s} =  \tau^\#
    \]
    are conjectured in \cite{DISST22}. Our arugment shows the implication 
    \[
 q_M(K) > 0 \Rightarrow d(S^3_{1} (K)) <0 
    \]
    does not hold. This is the first slice torus invariant from gauge theory which does not satisfy this kind of implication. 
\end{rem}

\section{$q_M$ and $\theta$ for prime knots with small crossing number }\label{qmtehta}
Let us review constructions of invariants $q_M$ and $\theta$ and summarize their basic properties.

Let $p$ be a prime number.
For a knot $K $ in $S^3$, the associated {\it Seiberg--Witten Floer homotopy type} introduced by Manolescu \cite{Man03} to its 
 $\Z_p$-branched cover $\Sigma_p(K)$ with unique $\Z_p$-invariant spin structure
\[
SWF (\Sigma_p(K)) 
\]
has a $S^1 \times \Z_p$-action \footnote{Strictly speaking, $S^1\times \Z_p$-equivariant Seiberg--Witten Floer stable homotopy type which is independent of the Riemannian metric is not formulated, while the equivariant Seiberg--Witten Floer cohomology which is independent of the Riemannian metric is formulated. We use the latter only. }. 
Baraglia--Hekmati \cite{BH,BH2} and Baraglia \cite{Bar} studied its $S^1\times \Z_p$-equivariant cohomology
\[
H^*_{S^1\times \Z_p} (SWF (\Sigma_p(K)); \mathbb{F}_p )
\]
equipped with the module structure over $H^*(B(S^1\times \Z_p); \mathbb{F}_p)$ and introduced a family of concordance invariants \[
\theta^{(p)}(K) \in \frac{1}{p-1} \Z_{\geq 0} 
\]
from the module structure of $H^*_{S^1\times \Z_p} (SWF (\Sigma_p(K)) )$. 
They reproved and the Milnor conjecture using the 4-genus bound comes from the invariants $\theta^{(p)}(K)$. 
For the computations of $\theta^{(p)}(K)$, they have used: 
\begin{itemize}
\item 
the spectral sequence 
\[
E_2=H^*(B\Z_p; \tilde{H}^*_{S^1}(SWF(\Sigma_p(K)); \mathbb{F}_p) \Rightarrow H^*_{S^1\times \Z_p}(SWF(\Sigma_p(K)); \mathbb{F}_p), 
\]
\item
the isomorphism between the $S^1$-equivariant monopole Floer homologies and the Heegaard Floer homology $HF^+$ , and
\item 
the graded root technique to compute the Heegaard Floer cohomology of almost rational 3-manifolds.
\end{itemize}
 The first and fourth authors studied $p=2$ case and found that $\Z_2$-equivariant cohomology (forgetting the $S^1$ action!) gave a $\Z$-valued slice-torus invariant  $q_M$  and in particular gave an alternative proof of the Milnor conjecture.  The authors did not use  Baraglia--Hekmati's tools mentioned above, but instead used the following methods to show that $q_M$ is actually a slice-torus invariant.
\begin{itemize}
\item 
Some methods used by Kronhimer \cite{Kr97} and Daemi--Scaduto \cite{DS19}\cite{DS23} for singular instanton theory based on Freedman--Quinn's work \cite{FQ90} on normally immersed surface in four-manifolds. This was used to show \[
\operatorname{rank}_{\mathbb{F}_2[Q]} \wt{H}^*_{\Z_2}(SWF (\Sigma_2(K)); \mathbb{F}_2)=1
\]
and the "cobordism inequality" \[
q_M(K_1)\leq q_M(K_0)+g(S)
\]
for a smooth and orientable surface cobordism $S: K_0 \to K_1$ in $[0, 1]\times S^3$, which immediately implies the concordance invariance and the 4-ball genus bound for $q_M$.
\item 
The homotopical transverse knot invariant
\[
\Psi_2(K) : S^0 \to \Sigma^{\text{\textbullet}} SWF(-\Sigma_2(K))  
\]
introduced in \cite{IT24}, its gluing property and non-vanishing property, which is nothing but the equivariant version of the
stable homotopy version of the monopole contact invariant \cite{IT20} and its corresponding properties previously developed by the authors mainly in \cite{IT20}.
This is used to show the "adjunction equality" \[
q_M(\partial S)=g(S)
\]
for symplectic surface $S$ 
in $D^4$ with transverse knot boundary, which in particular implies the  computation for torus knots. 
A key point in the proof of the "adjunction equality" is to prove that the relative invariant $BF^*_{\Sigma_2(S)}(1)$ for such a symplectic surface  $S$ is non-$Q$-torsion and non-divisible by $Q$, and thus attains the "bottom of the $Q$ tower".
\end{itemize}
\par
Based on these methods, in \cite[Theorem 1.16]{IT24}, the authors proved 
\[
\operatorname{rank}_{\mathbb{F}_2[Q]}\wt{H}^*_{\Z_2}(SWF (\Sigma_2(K)); \mathbb{F}_2) =1, 
\]
where $Q$ is the degree one variable so that $\mathbb{F}_2[Q] = H^*(B\Z_2; \mathbb{F}_2)$.  
In addition, the module $\wt{H}^*_{\Z_2}(SWF (\Sigma_2(K)); \mathbb{F}_2)$ has an absolute $\Q$-grading.  

\begin{defn}
From this module, we define a concordance invariant
\[
q_M(K) := \min \{ i|    x \in \wt{H}^i_{\Z_2} (SWF(-\Sigma_2(K)); \mathbb{F}_2),  \, Q^n x \neq 0  \text{ for all }n\geq 0\} - \frac{3}{4} \sigma(K) \in \Z
\]
for a given knot $K$ in $S^3$. 
\end{defn}

\begin{rem}
Conjectually, the invariant $q_M$ is equal to Hendricks--Lipshitz--Sarkar's $q_\tau$ invariant \cite{HLS16} with a signature correction term, defined using $\Z_2$-equivariant Heegaard Floer homology of $\Sigma_2(K)$. A Heegaard Floer counterpart of the transverse knot invariant in \cite{IT24} has been considered in \cite{Ka18} by Kang. 
\end{rem}

From now on, we give the values of $q_M$ and $\theta$ for all 85 prime knots with crossing number $\leq 9$ and explain the current situation for those with crossing number $10$.
First of all, we summarize the basic properties of $q_M$ and $\theta=\theta^{(p=2)}$, which are proven in \cite{Ba22, IT24}. 
\begin{thm}
The invariants $q_M$ and $\theta$ have the following properties: 
\begin{enumerate}
\item 
For any knot $K \subset S^3$, we have integers
\[
q_M(K) \in \Z \text{ and } \theta(K) \in \Z_{\geq 0} , 
\]
which do not depend on orientations of $K$. Also, these are concordance invariants.  
\item 
For oriented knots $K_0, K_1 \subset S^3$, we have 
\[
q_M(K_0 \# K_1)=q_M(K_0)+q_M(K_1) \text{ and } \theta ( K_0 \# K_1) \leq \theta (K_0) + \theta (K_1). 
\]
\item 
For any knot $K \subset S^3$, we have
\[
q_M(K)\leq \theta(K) \leq g_4(K) \text{ and }- \frac{1}{2}\sigma(K)\leq \theta(K) . 
\]
\item 
For any transverse knot $\mathcal{T}\subset (S^3, \xi_{std})$, we have the Bennequin type inequality
\[
sl(\mathcal{T})\leq 2q_M(\mathcal{T})-1, 
\]
where $\xi_{std}$ is the unique tight contact structure on $S^3$.
Moreover, if $\mathcal{T}$ is a boundary of a connected symplectic surface $S \subset (D^4, \omega_{std})$, the equality
\[
sl(\mathcal{T})= 2q_M(\mathcal{T})-1=2g(S)-1
\]
holds, where $\omega_{std}$ denotes the standard symplectic structure on $D^4$ and $sl$ denotes the self-linking number.
As a consequence, for the quasipositive knot $K$, we have
\[
q_M(K)=\theta(K)=g_4(K).
\]
\item 
For a knot $K$ such that the double brached covering $\Sigma_2(K)$ is an $L$-space over $\mathbb{F}_2$ coefficient, 
\[
q_M(K)=\theta(K)=-\frac{\sigma(K)}{2}
\]
holds.
In particular, this relation holds for all quasialternating knots, since the branched double covering of a quasialternating knot is an $L$-space \cite{OS05}.

\item If $d ( \Sigma_2 (K)) < - \sigma(K)$ and $\sigma(K) \leq 0$, then we have 
\[
1 + \frac{1}{2}\sigma(K) \leq \theta(K) , 
\]
where the $d ( \Sigma_2 (K))$ denotes the $d$-invariant with the unique spin structure on $\Sigma_2 (K)$. 
\end{enumerate}
\end{thm}
The tables in \cref{Table} give the values of $q_M$ and $\theta$ for all 85 prime knots with crossing number $\leq 9$.
Notice that the values of $\theta$ for the knots with the opposite chirality are not calculated.
We compute these for a choice of chirality of $K$ such that $\sigma(K)\leq 0$.
The values of $-\sigma/2$ and $g_4$ are quoted from the knotinfo.
The values of $q_M$ and $\theta$ can be seen as follows.
Among the  85 prime knots with crossing number $\leq 9$, which are listed above,  all but \[
8_{19}, 9_{42}, 9_{46}
\]
are quasialternating and thus satisfy $q_M=\theta=-\sigma/2$.
Since $9_{46}$ is slice,  $q_M=\theta=-\sigma/2=0$.

For $8_{19}$,   $q_M=\theta=-\sigma/2$ still holds since $\Sigma_2(8_{19})=\Sigma(2, 3, 4)$ has a positive scalar curvature metric and thus an $L$-space as explained in the previous section.
We could also use the fact that $8_{19}=T_{3, 4}$ is quasipositive and thus $q_M=\theta=g_4=\frac{\overline{sl}+1}{2}$.
As stated by Greeene and seen in this paper, $\Sigma_2(9_{42})$ is also an L-space  over $\mathbb{F}_2$ and thus $q_M=\theta=-\sigma/2$ holds for $9_{42}$ as well.
\begin{rem}\label{10}
In this remark, we explain the circumstance for the 165 prime knots with corssing number 10, which can be read from the knot info.
Among them, all but 
\[
10_{124}, 10_{128}, 10_{132}, 10_{136}, 10_{139}, 10_{140}, 10_{145}, 10_{152}, 10_{153}, 10_{154}, 10_{161}
\]
are quasialternating and thus satsify $q_M=\theta=-\sigma/2$.
Among the remaining 11 knots, 
\[
10_{124}, 10_{128}, 10_{140}, 10_{145}, 10_{152}, 10_{154}, 10_{161}
\]
are quasipositive and thus satisfy  $q_M=\theta=g_4=\frac{\overline{sl}+1}{2}$.
The remaining 4-knots are 
\[
10_{132}, 10_{136}, 10_{139}, 10_{153}.
\]

Since $10_{153}$ is slice, we have $q_M(10_{136})=\theta(10_{136})=0$.
As for $10_{139}$, by \cite[Lemma 2.16]{FLL22}, $10_{139}$ is a squeezed knot, so $q_M=s/2=\tau=4$ holds for $10_{139}$.
Since $g_4(10_{139})=4$ is also known, we know $\theta(10_{139})=4$.
For $10_{132}$, we have $-\sigma/2=0$ and $g_4=1$, so we know $0\leq \theta(10_{132})\leq 1$ and $ q_M(10_{132})\leq \theta(10_{132})\leq 1$ but the authors could not determine these.
For $10_{136}$, $-\sigma/2=g_4=1$ holds, thus $\theta(10_{136})=1$ but $q_M(10_{136})$ is unknown.
\end{rem}

\begin{rem}
In all prime knots in which we know the values of invariants $q_M$ and $\theta$ in the tables, 
\[
|\sigma(K)|/2\leq |q_M(K)| \text{ and } q_M(K)=\theta(K)\]
hold. 
However, neither of these do not hold in general for a knot $K \subset S^3$. Set $K_0=-T(3,11) \# \#_{10} T(2,3)$, where $-K$ denotes the concordance inverse of $ K$.
We have
\[
q_M(T(3, 11))=g_4( T(3,11) ) = 10, \ 
\sigma ( T(3,11) )  = -16, \  
q_M(T(2, 3))=g_4 (T(2,3) ) =1,  \  
\sigma ( T(2,3) )  = -2 . 
\]
Now by the additivity of $q_M$ and the signature under the connected sum operation, we have 
$q_M ( K_0) =0 $
$\sigma (K_0) = -4 $. 
Thus, $2=|\sigma(K_0)|/2> |q_M(K_0)|=0$. 
Moreover,  we have $\sigma(K)/2\leq \theta(K)$, so 
\[
0=q_M(K_0)< -\sigma(K_0)/2=2\leq\theta(K_0).
\]
Furthermore, by considering $\#_n K_0$, we can see that the difference $\theta(K)-q_M(K)$ can be arbitrarily large.
\end{rem}
\begin{ques}

Find  $q_M(10_{132})$, $\theta(10_{132})$, and $q_M(10_{136})$.
As detailed in \eqref{10}, these are the only undetermined values among the prime knots with the crossing number $\leq 10$.
Since $10_{132}$ and $10_{136}$ are Montesinos knots, the authors expect that these can be computed by using the Seifert fibered structures on the double-branched coverings.
\end{ques}

 \clearpage
\section{Tables}\label{Table}
\begin{center}
\begin{tabular}{ c c c c c c c }
 $K$ &$-\sigma(K)/2$ & $q_M(K)$ & $\theta(K)$ & $g_4(K)$ &(quasi)alternating & positivity($BP\subset P \subset SQ \subset QP$)  \\ 
  $0_1$ & $0$ & $0$ & $0$ & $0$&  alt& BP \\ 
  $3_1$ & $1$ & $1$ & $1$ & $1$&  alt& BP\\ 
  $4_1$ & $0$ & $0$ & $0$ & $1$&  alt&-  \\
  $5_1$ & $1$ & $1$ & $1$ & $1$&  alt& BP\\
  $5_2$ & $1$ & $1$ & $1$ & $1$&  alt& P\\ 
  $6_1$ & $0$ & $0$ & $0$ & $0$&  alt& -\\
  $6_2$ & $1$ & $1$ & $1$ & $1$&  alt & -\\
  $6_3$ & $0$ & $0$ & $0$ & $1$&  alt &- \\ 
  $7_1$ & $3$ & $3$ & $3$ & $3$&  alt& BP\\
  $7_2$ & $1$ & $1$ & $1$ & $1$&  alt& P\\
  $7_3$ & $2$ & $2$ & $2$ & $2$& alt& P\\
  $7_4$ & $1$ & $1$ & $1$ & $1$& alt& P\\
  $7_5$ & $2$ & $2$ & $2$ & $2$&  alt& P\\
  $7_6$ & $1$ & $1$ & $1$ & $1$& alt&- \\
  $7_7$ & $0$ & $0$ & $0$ & $1$& alt&-\\
  $8_1$ & $0$ & $0$ & $0$ & $1$& alt &-\\
  $8_2$ & $2$ & $2$ & $2$ & $2$& alt &-\\
  $8_3$ & $0$ & $0$ & $0$ & $0$& alt &-\\
  $8_4$ & $1$ & $1$ & $1$ & $1$& alt &-\\
  $8_5$ & $2$ & $2$ & $2$ & $2$& alt &-\\
  $8_6$ & $1$ & $1$ & $1$ & $1$& alt &-\\
  $8_7$ & $1$ & $1$ & $1$ & $1$& alt &-\\
  $8_8$ & $0$ & $0$ & $0$ & $0$& alt &-\\
  $8_9$ & $1$ & $1$ & $1$ & $1$& alt &-\\ 
  $8_{10}$ & $1$ & $1$ & $1$ & $1$& alt &-\\
  $8_{11}$ & $1$ & $1$ & $1$ & $1$& alt &-\\
  $8_{12}$ & $0$ & $0$ & $0$ & $1$& alt &-\\
  $8_{13}$ & $0$ & $0$ & $0$ & $1$& alt &- \\
  $8_{14}$ & $1$ & $1$ & $1$ & $1$& alt &-\\
  $8_{15}$ & $2$ & $2$ & $2$ & $2$& alt&P \\
  $8_{16}$ & $1$ & $1$ & $1$ & $1$& alt &-\\
  $8_{17}$ & $0$ & $0$ & $0$ & $1$& alt &-\\
  $8_{18}$ & $0$ & $0$ & $0$ & $1$& alt &-\\
  $8_{19}$ & $3$ & $3$ & $3$ & $3$&  non-q.alt!& BP\\
  $8_{20}$ & $0$ & $0$ & $0$ & $0$&  q.alt& QP\\
  $8_{21}$ & $1$ & $1$ & $1$ & $1$&  q.alt &- 
     \end{tabular}
  \newpage
  \begin{tabular}{ c c c c c c c }
  $K$ &$-\sigma(K)/2$ & $q_M(K)$ & $\theta(K)$ & $g_4(K)$ &(quasi)alternating & positivity ($BP\subset P \subset SQ \subset QP$)\\ 
  $9_1$ & $4$ & $4$ & $4$ & $4$& alt& BP\\
  $9_2$ & $1$ & $1$ & $1$ & $1$&  alt& P\\
  $9_3$ & $3$ & $3$ & $3$ & $3$&  alt& P\\
  $9_4$ & $2$ & $2$ & $2$ & $2$& alt& P\\
  $9_5$ & $1$ & $1$ & $1$ & $1$& alt& P\\
  $9_6$ & $3$ & $3$ & $3$ & $3$& alt& P\\
  $9_7$ & $2$ & $2$ & $2$ & $2$& alt& P\\
  $9_8$ & $1$ & $1$ & $1$ & $1$& alt&-\\
  $9_9$ & $3$ & $3$ & $3$ & $3$& alt & P \\
  $9_{10}$ & $2$ & $2$ & $2$ & $2$& alt& P \\
  $9_{11}$ & $2$ & $2$ & $2$ & $2$& alt &-\\
  $9_{12}$ & $1$ & $1$ & $1$ & $1$& alt &P\\
  $9_{13}$ & $2$ & $2$ & $2$ & $2$& alt &-\\
  $9_{14}$ & $0$ & $0$ & $0$ & $1$& alt &-\\
  $9_{15}$ & $1$ & $1$ & $1$ & $1$& alt &-\\
  $9_{16}$ & $3$ & $3$ & $3$ & $3$& alt& P\\
  $9_{17}$ & $1$ & $1$ & $1$ & $1$& alt &-\\
  $9_{18}$ & $2$ & $2$ & $2$ & $2$& alt& P\\
  $9_{19}$ & $0$ & $0$ & $0$ & $1$& alt &-\\
  $9_{20}$ & $2$ & $2$ & $2$ & $2$& alt &-\\
  $9_{21}$ & $1$ & $1$ & $1$ & $1$& alt &-\\
  $9_{22}$ & $1$ & $1$ & $1$ & $1$& alt &-\\
  $9_{23}$ & $2$ & $2$ & $2$ & $2$& alt& P\\
  $9_{24}$ & $0$ & $0$ & $0$ & $1$& alt&-\\
  $9_{25}$ & $1$ & $1$ & $1$ & $1$& alt &-\\
  $9_{26}$ & $1$ & $1$ & $1$ & $1$& alt &-\\
  $9_{27}$ & $0$ & $0$ & $0$ & $0$& alt &-\\
  $9_{28}$ & $1$ & $1$ & $1$ & $1$& alt &-\\
  $9_{29}$ & $1$ & $1$ & $1$ & $1$& alt &-\\
  $9_{30}$ & $0$ & $0$ & $0$ & $1$& alt &-\\
  $9_{31}$ & $1$ & $1$ & $1$ & $1$& alt &-\\
  $9_{32}$ & $1$ & $1$ & $1$ & $1$& alt &-\\
  $9_{33}$ & $0$ & $0$ & $0$ & $1$& alt &-\\
  $9_{34}$ & $0$ & $0$ & $0$ & $1$& alt &-\\
  $9_{35}$ & $1$ & $1$ & $1$ & $1$& alt & P\\
  $9_{36}$ & $2$ & $2$ & $2$ & $2$& alt &-\\
  $9_{37}$ & $0$ & $0$ & $0$ & $1$& alt &-\\
  $9_{38}$ & $2$ & $2$ & $2$ & $2$& alt &-\\
  $9_{39}$ & $1$ & $1$ & $1$ & $1$& alt &-\\
  $9_{40}$ & $1$ & $1$ & $1$ & $1$& alt &-\\
  $9_{41}$ & $0$ & $0$ & $0$ & $0$& alt &-\\
  $9_{42}$ & $1$ & $1$ & $1$ & $1$& non-q.alt! &-\\
  $9_{43}$ & $2$ & $2$ & $2$ & $2$& q.alt &-\\
  $9_{44}$ & $0$ & $0$ & $0$ & $1$&  q.alt &-\\
  $9_{45}$ & $1$ & $1$ & $1$ & $1$&  q.alt &-\\
  $9_{46}$ & $0$ & $0$ & $0$ & $0$& non-q.alt! &-\\
  $9_{47}$ & $1$ & $1$ & $1$ & $1$& q.alt &-\\
  $9_{48}$ & $1$ & $1$ & $1$ & $1$&  q.alt &-\\
  $9_{49}$ & $2$ & $2$ & $2$ & $2$&  q.alt & P  
\end{tabular}
\end{center}

(alt=alternating, \quad q.alt=quasi alternating, \quad BP=braid positive,\quad P=positive,\quad SQ=strongly quasipositive, \quad QP=quasipositive)

\bibliographystyle{plain}
\bibliography{tex}

\end{document}